\documentclass{article}

\usepackage{natbib}
\setcitestyle{authoryear,open={(},close={)}} 

\usepackage{float}
\usepackage{xcolor}
\usepackage{graphicx}
\usepackage{hyperref}
\usepackage{amsmath, amsthm, amssymb, amsfonts}
 \usepackage{amsfonts}
 \usepackage{amsmath}

\DeclareMathOperator*{\argmin}{arg\,min}
 \usepackage{amsthm}
 \usepackage{graphicx}
 \usepackage{hyperref}
 \usepackage{bigints}
 \usepackage{wrapfig} 
 \usepackage{caption}
\usepackage{subcaption}
\usepackage{mathtools}

\usepackage{changes}

\definechangesauthor[name=Soe, color=blue]{so}
\definechangesauthor[name=Ern, color=violet]{em}

\usepackage[english]{babel}
\newtheorem{thm}{Theorem}[section] 
\newtheorem{theorem}[thm]{Theorem}
\newtheorem{assumptions}[thm]{Assumption}

\newtheorem{corollary}[thm]{Corollary}
\newtheorem{lemma}[thm]{Lemma}
\newtheorem{proposition}[thm]{Proposition}
\newtheorem{remark}[thm]{Remark}
\newtheorem{definition}[thm]{Definition}

\newcommand\var{\operatorname{\mathbf{var}}}

\newcommand{\R}{\mathbb{R}}

\renewcommand\P{\operatorname{\mathbf{P}}}
\newcommand\E{\operatorname{\mathbf{E}}}
\makeatletter
\newcommand*\bigcdot{\mathpalette\bigcdot@{.5}}
\newcommand*\bigcdot@[2]{\mathbin{\vcenter{\hbox{\scalebox{#2}{$\m@th#1\bullet$}}}}}
\makeatother

\begin{document} 
 \title{Two sided ergodic singular control and \\ mean-field game for diffusions}
 \author{S\"oren Christensen\footnote{Department of Mathematics, Christian-Albrechts-Universit\"at Kiel, Germany. email: christensen@math.uni-kiel.de} \and Ernesto Mordecki\footnote{Centro de Matem\'atica. Facultad de Ciencias, Universidad de la Rep\'ublica, Montevideo, Uruguay. Corresponding author. email: mordecki@cmat.edu.uy} \and Facundo Oli\'u\footnote{Ingenier\'ia Forestal. Centro Universitario de Tacuaremb\'o, Universidad de la Rep\'ublica, Uruguay. email: facundo.oliu@cut.edu.uy}}
\maketitle
\begin{abstract} %
In a probabilistic mean-field game driven by a linear diffusion 
an  individual player aims to minimize an ergodic long-run cost by
controlling the diffusion through a pair of --increasing and decreasing-- c\`adl\`ag processes, 
while he is interacting with an aggregate of players through the expectation of a similar diffusion controlled by another pair of c\`adl\`ag processes.
In order to find equilibrium points in this game, 
we first consider the control problem, in which the individual player has no interaction with the aggregate of players.
In this case, we prove that the best policy is to reflect the diffusion process within two thresholds. 
Based on these results, we obtain criteria for the existence of equilibrium points in the mean-field game in the case when the controls of the aggregate of players are of reflection type, and give a pair of nonlinear equations to find these equilibrium points.
In addition, we present an approximation result for Nash equilibria of erdogic games with finitely many players to the mean-field game equilibria considered above when the number of players tends to infinity.
These results are illustrated by several examples where the existence and uniqueness of the equilibrium points depend on the coefficients of the underlying diffusion.
\end{abstract}
\vspace{1cm}
\newpage
\section{Introduction}
In recent years, mean-field game theory has emerged as a powerful framework for modeling the behavior of large populations of interacting players in a stochastic environment. This interdisciplinary field lies at the intersection of mathematics, economics, and engineering, offering deep insights into complex systems characterized by strategic interactions. 
Mean-field game models have found applications in various domains, including for instance, finance, energy systems \citep{Carmona}, or traffic management and social dynamics \citep{ADRIANOFIESTA!}.
The two seminal papers in the field can be considered the contributions by 
\citet{HMC(2006)} and \citet{LL(2007)}. The key issue in their proposals, 
under the assumption of a large number of identically interacting players, 
is that individual actions do not affect a mean state of the system. 
This means that an individual player faces an optimization problem against a synthetic player, 
resulting from the aggregation of a large number of players, which is referred to in this paper as \emph{the market}.
The success of the proposal made it possible to solve various problems, many of which can be found in the two-volume monograph by \cite{CD(2018)}, which has become a central reference in the field.

The first results of the present paper are in the framework of singular control of diffusions.
Our departure point are the results by \cite{Alvarez}, where the existence and uniqueness of optimal reflecting controls for a diffusion are established. 
Our contribution is to extend these results to show that the solution found by \cite{Alvarez} is in fact the optimal control within the larger class of finite variation controls.
To do this, we use the solution of the two-sided ergodic singular control, in the framework proposed by \cite{Alvarez}, thus extending the class of controls.
To achieve our goals, we postulate a verification result in the form of a Hamilton-Jacobi-Bellman equation and use the ergodic properties of the controlled processes to obtain an analytic problem. 
The control problem has been studied extensively in the literature, see for example \cite{AS}, \cite{HNUW}, and \cite{LESE}. 
With respect to applications of singular control results, we mention studies focusing on cash flow management that investigate optimal dividend distribution, recapitalization, or a combination of both, while considering risk neutrality. 
See, for example, \cite{AT}, \cite{HT}, \cite{JS}, \cite{P}, \cite{BK}, and \cite{SLG}. 

Our second aim in the present paper is to incorporate a mean-field game dependence into the two-sided ergodic singular control problem for It\^o diffusions just described.
%
As a consequence, we obtain necessary and sufficient conditions for the existence of mean-field game equilibrium points,
and,  for more restricted families of cost functions, uniqueness within the class of reflecting strategies. 
Finally, we define an $N$-player problem and prove that a mean-field equilibrium is an approximate Nash equilibrium for the $N$-player game.

The mean-field game framework is less discussed in the literature.  However, there has been increased activity in this area in the recent past.
Here we would like to mention the current papers by \cite{ABF,CDF,DFT,KKXYZ,SBT,CG}, on the explicit solution of stationary Mean Field Games with singular and impulsive controls. 

The rest of the paper is organized as follows. 
In Section \ref{S:theproblem} we study the control problem. 
After introducing the necessary tools, 
we state and prove the main result of the section, i.e. the optimality of reflecting controls obtained within the class of càdlàg controls.
In Section \ref{S:MFG} we consider the mean-field game problem. It adds the complexity of a two-variable cost function where the second variable represents the market. The main result consists of a set of conditions for the existence and uniqueness of equilibrium strategies, containing also a particular analysis when the cost function is multiplicative. 
Section \ref{S:examples} presents three examples that illustrate these results.
Section \ref{S:epsilon} contains approximation results. The equilibrium found for mean field games,
becomes the limit of Nash equilibrium strategies 
when considering an individual player in the framework of a symmetric $N$-player game.
A final appendix includes some auxiliary computations corresponding to the examples of Section \ref{S:examples}.
\section{Control problem}\label{S:theproblem}
In this section we consider the one-player control problem.
We first recall results obtained by \citet{Alvarez} that play a fundamental role along the paper. 
These results consist in the determination of optimal control levels in an ergodic framework 
for a diffusion within the class of reflecting controls. 
We then prove that the optimal levels found in \cite{Alvarez} in fact give the optimal controls within the broader class of
of finite variation c\`adl\`ag controls.
\subsection{Diffusion}
Let us consider a filtered probability space 
$(\Omega,\mathcal{F},\lbrace \mathcal{F}_t\colon t\geq 0^-\rbrace,\P)$ 
that satisfy the usual assumptions.
In order to define the underlying diffusion consider the functions 
$\mu\colon\R\to\R$ and $\sigma\colon\R\to\R$ 
assumed to be locally Lipschitz. 
Under these conditions the stochastic differential equation 
\begin{equation}\label{sde}
dX_t=\mu (X_t)dt+ \sigma(X_t)dW_t , \ X_{0}=x_0
\end{equation}
has a unique strong solution up to an explosion time, 
that we denote by $X=\{X_t\colon t\geq 0^-\}$ (see \citet[Theorem V.38]{ProtterPE}).
Observe that our framework includes quadratic coefficients. 

As usual, we define the infinitesimal generator of the process $X$ as
$$
\mathcal{L}_X = \frac{1}{2} \sigma^2 (x) \frac{\displaystyle d^2}{\displaystyle d^2 x} + \mu(x)\frac{\displaystyle d}{\displaystyle dx}. 
$$
We denote the density of the scale  function $S(x)$ w.r.t the Lebesgue measure as
$$
S'(x)=\exp \left(-\int^x \frac{2\mu(u)}{\sigma^2(u)} du \right),
$$
and the density of the speed measure $m(x)$ w.r.t the Lebesgue measure as 
$$
m'(x)=\frac{2}{\sigma^2(x)S'(x)}.
$$
As mentioned above, the underlying process is controlled by a pair of processes, the admissible controls, that drive it to a convenient region, defined below. 
\begin{definition}\label{D:admissiblecontrols} 
An \emph{admissible control} is a pair of non-negative $\lbrace \mathcal{F}_t \rbrace$-adapted processes 
$\eta=(U=\{U_t\}_{t\geq 0^-},D=\{D_t\}_{t\geq 0^-})$ such that:
\begin{itemize}
\item[\rm(i)] Each process $U,D$ is right continuous and non decreasing almost surely.
\item[\rm(ii)] For each $t\geq 0^-$ the random variables $U_t$ and $D_t$ have finite expectation.
\item[\rm(iii)] For every $x \in \R$ the stochastic differential equation
\begin{equation}\label{sdec} 
dX^\eta_t:= \mu (X^\eta_t)dt + \sigma(X^\eta_t)dW_t +dU_t-dD_t, \quad X_{0^-}=x
\end{equation}
has a unique strong solution with no explosion in finite time. 
\end{itemize}
We denote by $\mathcal{A}$ the set of admissible controls.
\end{definition}
Note that condition (iii) is satisfied, for instance, when the coefficients are globally Lipschitz. 
(See the remark after Theorem V.38 in \cite{ProtterPE}.) 
Observe also that condition (ii) is not a real restriction, as, for instance, the integral in the cost function $G(x)$ in \eqref{ecf}
that we aim to minimize, in case of having infinite expectations, is infinite.  
A relevant sub-class of admissible controls is the set of reflecting controls.
\begin{definition}\label{D:reflectingX} 
For $a<b$ denote by $X^{a,b}=\lbrace X^{a,b}_t\colon t \geq 0 \rbrace$ 
the strong solution 
of the stochastic differential equation with reflecting boundaries at $a$ and $b$: 
$$
dX^{a,b}_t= \mu (X^{a,b}_t)dt + \sigma(X^{a,b}_t)dW_t+dU^a_t-dD^b_t , \qquad X_{0^-}=x. 
$$
Here $U^a=\{U^a_t\}, D^b=\{D^b_t\}$, are the local times of the reflected diffusion in the interval $[a,b]$. 
They are continuous non-decreasing processes that increase, respectively, only when the solution visits $a$  or $b$,
and make the controlled diffusion satisfy the condition $a\leq X^{a,b}_t\leq b$, a.s. for all 
$t\geq 0$. 
As the above equation has a strong solution (see \cite[Theorem 5.1]{saisho}), 
the pair $(U^a, D^b)$ 
belongs to $\mathcal{A}$, we call them \emph{reflecting controls}. 
If $x \notin (a,b)$, we begin the policy by sending the process to the closest point of the interval 
$[a,b]$
at time $t=0$. This is why we need to begin our evolution at $t=0^-$, in order to have c\`adl\`ag controls. 
\end{definition}
We introduce below the cost function $c(x,y)$ to be considered in the mean-field game formulation,
satisfying some natural conditions.
\begin{assumptions}\label{A:Assumptions2}
Assume that $c \colon \mathbb{R}^2 \rightarrow \mathbb{R}_+$ is a continuous function, 
and the positive constants $q_u, q_d$ are the unit cost of using the associated controls. 
Assume that, for each fixed $y \in \mathbb{R}$ there exist a value $x_y$  such that
\begin{equation*}
c(x,y)  \geq c(x_y,y)\geq 0, \qquad\text{for all $x \in \R$},
\end{equation*}
and positive constants $K_y$ and $\alpha_y$ such that
\begin{equation}\label{E:kalpha}
c(x,y)+K_y \geq \alpha_y \vert x \vert, \qquad \text{for all $x \in \R$}.
\end{equation}
Consider the maps
\begin{equation*}
\pi_1(x,y)=c(x,y)+q_d\mu(x), \hspace{10mm} \pi_2(x,y)=c(x,y)-q_u \mu(x),
\end{equation*}
and assume that for each fixed $y \in \mathbb{R}$:
\begin{itemize}
\item[\rm(i)]
There exists a unique real number $x^y_i = \argmin\lbrace \pi_i(x,y)\colon x\in\R \rbrace$ so that $\pi_i( \cdot ,y)$ is decreasing on  $(-\infty ,x^y_i)$  and
increasing on $(x^y_i, \infty)$, where $i=1,2$.
\item[\rm(ii)] The following limits hold:
\begin{equation}\label{E:pinecesidad}
\lim_{x\rightarrow \infty} \pi_1(x,y) = \lim_{x \rightarrow -\infty} \pi_2(x,y) = \infty.
\end{equation}
\end{itemize} 
\end{assumptions}
\begin{remark}
In the control problem case, when there is no aggregate of players, the cost function depends only on the first variable.
We then set the second variable above  to $y=0$, and denote $c(x)=c(x,0)$. 
This function satisfies $c(x)\geq c(x_0)\geq 0$ for some $x_0$,
$c(x)+K\geq\alpha |x|$, for some positive constants $K$ and $\alpha$, 
and the functions are $\pi_1(x)=c(x)+q_d\mu(x)$ and $\pi_2(x)=c(x)-q_u\mu(x)$
have their respective minima at $x^0_1,x^0_2$, and satisfy \eqref{E:pinecesidad}.
\end{remark}
\begin{definition}
We define the ergodic cost function as
\begin{equation}\label{ecf}
G(x)= \inf_{\eta \in \mathcal{A}} \limsup_{T\to\infty} \frac{1}{T} \E_x \left(  \int_0^T c(X^\eta_s)ds +q_u U_T +q_d D_T \right),  
\end{equation}
where $\eta=(U,D)$ is an admissible control in $\mathcal{A}$.
\end{definition}
The existence of a unique pair of optimal controls within the class of reflecting controls was obtained by \citet{Alvarez},
from where we borrow the notation and assumptions.
In the following result, we summarize (in a convenient way for our purposes) 
results of Lemma 2.1 and Theorem 2.3 from \cite{Alvarez}. 
Let us mention that condition \eqref{E:kalpha} is not necessary for \cite{Alvarez} results, 
we will use it in the sequel to prove optimality within the class of feasible controls.
\begin{theorem}[\cite{Alvarez}] \label{T:alvarezcost} 
Under Assumption \ref{A:Assumptions2}: 
\par\noindent{\rm(a)}
 If $a<b$ then
\begin{multline}\label{eq:cab}
\lim_{T \rightarrow \infty} \frac{1}{T} \E_x \left( \int_0^T c(X_s^{a,b}) ds  +q_u U^a_T + q_d D^b_T \right) \\
 =\frac{1}{m(a,b)} \left[  \int_a^b c(u)m(du)  +\frac{q_u}{S'(a)}+\frac{q_d}{S'(b)}    
\right]=:C(a,b). 
\end{multline}
\par\noindent{\rm(b)}
There is an unique pair of points $a^{\ast}<b^{\ast}$ that  satisfy the equations:   
\begin{itemize}
\item[\rm(i)] $\pi_1(b^{\ast}) = \pi_2(a^{\ast})$, 
\item[\rm(ii)] $ \int_{a^{\ast}}^{b^{\ast}} \left( \pi_1(t)-\pi_1(b^{\ast}) \right) m(dt) + \frac{\displaystyle q_u+q_d}{\displaystyle S'(a^{\ast})}  =0$.
\end{itemize}
Furthermore, the pair $(a^{\ast}, b^{\ast})\in (-\infty,x_2^0) \times (x_1^0, \infty)$ minimizes the expected long-run average cost within the class of reflecting controls.
\end{theorem} 
\begin{remark}
Condition {\rm(i)} is obtained from the fact that $X^{a,b}$ is stationary. Regarding equation {\rm(ii)}, it arises after differentiation in order to determine the minimum. The uniqueness of the solution is proved based on the properties of the cost function. 
Conditions {\rm(i)} and {\rm(ii)} here are equivalent to conditions (2.5) and (2.6) in \cite{Alvarez}
as they reduce to solving $C(a,b)-\pi_2(a)=\pi_1(b)-C(a,b)=0$, as seen in the proof of Lemma 2.1 in \cite{Alvarez}. See more details in \cite{Alvarez}.
\end{remark}

\subsection{Optimality within $\mathcal{A}$}

Optimality within the class $\mathcal{A}$ of càdlàg controls requires further analysis. 
As expected, and mentioned in \cite{Alvarez}, 
the optimal controls within class $\mathcal{A}$ are the same controls found in the class of reflecting controls. 
An analogous result to the one presented below was obtained for non-negative diffusions when considering a maximization problem in \cite{CDF} (see also \cite{KKXYZ}). 
More precisely, it is clear that
$$ 
\inf_{a<b}  \lim_{T \rightarrow \infty} \frac{1}{T} 
\E_x \left( \int_0^T c(X_s^{a,b}) ds +q_u U^b_T+q_d D^a_T  \right)  \geq G(x).
$$
Then, to establish the optimality within $\mathcal{A}$ it is necessary to obtain the other inequality.
This task is carried out with the help of the solution of the free boundary problem \eqref{free} below, similarly to \cite{CDF}. 
The mentioned differences with this situation require different hypotheses and slightly different arguments. 
\begin{theorem}[Verification]\label{T:Verification}
Consider a diffusion defined by \eqref{sde} and a cost function $c(x)$ satisfying Assumption \ref{A:Assumptions2}.
Suppose that there exist a constant $\lambda  \geq 0$ and a function $u \in C^2(\mathbb{R})$ such that
\begin{equation}\label{fbp}
(\mathcal{L}_X u) (x)+c(x) \geq \lambda, \qquad -q_u \leq u'(x) \leq q_d,\quad\text{for all $x\in\R$}. 
\end{equation}
Define the subset of 
admissible controls
\begin{equation}\label{liminf}
    \mathcal{B}= \left\lbrace \eta \in \mathcal{A}\colon \liminf_{T \rightarrow \infty} \frac{1}{T}\left|\E_x(u(X^\eta_T))\right|=0 \right\rbrace.
\end{equation}
Then, 
\begin{equation}\label{ergodic}
\inf_{\eta \in \mathcal{B}} \limsup_{T\to\infty} \frac{1}{T} \E_x \left(  \int_0^T c(X^\eta_s)ds +q_u U_T +q_d D_T \right)  
\geq\lambda.
\end{equation}
\end{theorem}
\begin{remark}
The consideration of the subclass $\mathcal{B}$ is not a restriction, as will be seen below. 
More precisely, it will be proved (using condition \eqref{E:kalpha}), 
that controls in
$\mathcal{A}\setminus\mathcal{B}$ give infinite values of the long run costs, 
being then not relevant in 
the computation of $G(x)$ in \eqref{ecf}.
\end{remark}
\begin{proof} Fix $T>0$.
For each $n\geq 1$ define the stopping times 
\[ 
T_n= \inf \lbrace t \geq 0 \colon \vert X^\eta_t \vert \geq n \rbrace\wedge T\nearrow T\quad a.s.
\] 
Using It\^{o} formula for processes with jumps 
(observe that the diffusion $X$ is continuous but the controls can have jumps, 
and in consequence the controlled processes $X^\eta$ can have jumps),
\begin{align}\label{E:verification1}
u(X^\eta(T_n)) & = u(x)+ \int_0^{T_n }u'(X^\eta_{s-} )  dX^\eta_s 
 +\frac{1}{2} \int_0^{T_n}u''(X^\eta_{s-})d \langle (X^\eta)^c,(X^\eta)^c \rangle _s \nonumber  \\
 &\quad +\sum_{s \leq T_n} \left( u(X^\eta_s)-u(X^\eta_{s-})   -u'(X^\eta_{s-} ) \bigtriangleup X^\eta_s \right). 
\end{align}
The r.h.s in \eqref{E:verification1} can be rewritten as
\begin{align}\label{E:verification2}
 u(x)& + \int_0^{T_n} (\mathcal{L}_X u)(X^\eta_{s-})ds -\int_0^{T_n}\mu(X^\eta_{s-})u'(X^\eta_{s-})  ds  \nonumber \\
 &+ \int_0^{T_n}u'(X^\eta_{s-}) d X^\eta_s   
 +\sum_{s \leq T_n} \left( u(X^\eta_s)-u(X^\eta_{s-})   -u'(X^\eta_{s-} ) \bigtriangleup X^\eta_s \right). 
\end{align}
Using the fact that $u'(X^\eta_{s-})=u'(X^\eta_s)$ in a set of total Lebesgue measure in $[0,T]$ almost surely,
and that $\bigtriangleup X^\eta_s= \bigtriangleup U_s-\bigtriangleup D_s$, we rewrite \eqref{E:verification2} as 
\begin{align}\label{E:verification3} 
u(x) &+ \int_0^{T_n} (\mathcal{L}_X u)(X^\eta_{s-})ds + \int_0^{T_n}u'(X^\eta_{s-})\sigma(X^\eta_{s-}) d W_s  \nonumber \\
&+ \int_0^{T_n} u'(X^\eta_{s-}) d (U_s-D_s)\nonumber  \\
   &+\sum_{s \leq T_n}\left( u(X^\eta_s)-u(X^\eta_{s-})  - u'(X^\eta_{s-})(\bigtriangleup U_s -\bigtriangleup D_s) \right). 
\end{align}
Therefore, denoting by $U_s^c$ and $D_s^c$ the continuous parts of the processes $U_s$ and $D_s$ respectively, 
and using the inequalities \eqref{fbp} in the hypothesis,  we obtain  
\begin{align*}
u(X^\eta (T_n)) & \geq u(x)+ \lambda T_n - \int_0^{T_n}c(X^\eta_{s-}) ds  
 +\int_0^{T_n} u'(X^\eta_{s-}) \sigma (X^\eta_{s-}) dW_s\\&\quad
 -\int_0^{T_n}q_u dU^c_s  
 -\int_0^{T_n} q_d dD^c_s - \sum_{0 \leq s \leq T_n} ( \bigtriangleup U_s q_u+\bigtriangleup D_s q_d)  \\
 &=u(x)+ \lambda T_n - \int_0^{T_n}c(X^\eta_{s-}) ds +\int_0^{T_n} u'(X^\eta_{s-}) \sigma (X^\eta_{s-}) dW_s \\
  &\quad-q_u U_{T_n}-q_d D_{T_n}.
\end{align*}
Rearranging the terms above and taking the expectation we obtain
\begin{multline*}
\E_x (u(X^\eta (T_n))) -u(x)+ 
\E_x \left(\int_0^{T_n} c(X^\eta_{s-})ds + q_u U_{T_n}+ q_d D_{T_n} \right)  \geq \lambda \E_x (T_n).
\end{multline*}
Taking first limit as $n$ tends to infinity, dividing then  by $T$,  and finally taking $\liminf$ as $T$ goes to infinity we obtain \eqref{ergodic}
concluding the proof of the verification theorem.
\end{proof}
Consideration of free boundary problems such as \eqref{fbp} in the framework of singular control problems can be found for example in \cite{Alvarez}, \cite{CDF}, and \cite{KKXYZ}. 
In \cite{Alvarez}, the author studied the same problem of this section and used a free boundary problem to find some useful properties of optimal controls. 
More precisely, under the same assumptions as above, 
to study the ergodic optimal control problem in the class of reflecting controls, 
the author considered the free boundary problem consisting of finding $a<b, \lambda$ and a function $u$ in $C^2(\mathbb{R})$ such that
\begin{equation}\label{free}
\left\{
\aligned
(\mathcal{L}_X u) (x) + c(x)&=\lambda,&\text{for all $ x \in (a,b)$},\\ 
u(x)&=q_d (x-b) +u(b),&\text{for all $x \geq b$},\\
u(x)&=q_u (a-x) +u(a),&\text{for all $x \leq a$}.    
\endaligned
\right.
\end{equation}
For this problem, it is proved (see Remark 2.4 in \cite{Alvarez}) that there exists a unique solution that satisfies \eqref{fbp}. 
Furthermore, 
\begin{equation}\label{E:lambda}
    \lambda=C(a,b),
\end{equation}
the ergodic cost defined by \eqref{eq:cab}, 
as states equation (2.15) in \cite{Alvarez}.
Here, similar to \cite{KKXYZ}, we use the results obtained in \cite{Alvarez} to get a suitable candidate to apply Theorem \ref{T:Verification}. 
\begin{theorem}\label{T:Alvarezcompletation}
Consider a diffusion defined by \eqref{sde}
and a cost function $c(x)$ satisfying Assumption \ref{A:Assumptions2}.
Then, the reflecting controls with levels given in {\rm(b)} in Theorem \ref{T:alvarezcost} minimize the ergodic  cost  $G(x)$ in \eqref{ecf} within the set $\mathcal{A}$ of admissible controls.
\end{theorem}
\begin{proof}
Take $u$ as the solution of the free boundary problem \eqref{free} defined above. 
In view of Theorem \ref{T:Verification}, we need to prove that the infimum of the ergodic cost defining $G(x)$ is realized in the set $\mathcal{B}$ defined in \eqref{liminf}.
Take then  $\eta\in\mathcal{A}\setminus\mathcal{B}$.
By definition of $\mathcal{B}$,
there exist constants $\epsilon>0$ and $S>0$ such that
\begin{equation}\label{1s}
\E_x u(X^\eta_s)>\epsilon s,\quad\text{for all $s \geq S$}.
\end{equation}
The second statement in \eqref{fbp} implies that $|u(x)-u(0)|\leq (q_u+q_d)|x|$. From this, it follows
$$
c(x)\geq A u(x)-B,
$$
for $A=\alpha/(q_u+q_d)$ and $B=\alpha u(0)/(q_u+q_d)+K$, see \eqref{E:kalpha}.
In view of \eqref{1s}, this implies
\begin{equation*}
\limsup_{T\to\infty} \frac{1}{T} \E_x \left( \int_0^T c(X^\eta_s)ds \right) \geq   \limsup_{T\to\infty} 
\frac{1}{T} \int_S^T (A\epsilon s-B)ds =\infty.
\end{equation*}
As a consequence, for any $\eta\in\mathcal{A}\setminus\mathcal{B}$, we have 
$$
\limsup_{T \to \infty} \frac{1}{T}\E_x \left( \int_0^T c(X^\eta_s)ds+q_uU_T+q_d D_T\right) =\infty.
$$
Finally, as the class of reflecting controls gives finite ergodic limits by Theorem \ref{T:alvarezcost}, the infimum can be taken in the subclass $\mathcal{B}$.
So Theorem \ref{T:Verification} gives the equality $G(x)=\lambda=C(a,b)$ 
(see \eqref{E:lambda}), concluding the proof.
\end{proof}
\section{Mean-field Game Problem}\label{S:MFG}
As mentioned above, in the mean-field game formulation, 
the cost function depends on two variables, 
respectively the state of the player and the state of an aggregate of players referred to as \emph{the market}. 
The state of the market is the expectation of a continuous function of the diffusion process under some given controls.

The study of the existence and uniqueness of equilibrium points begins with the application of Theorem \ref{T:alvarezcost} 
when the state of the market is asymptotically constant. 
The cost function becomes one-dimensional and the results in \cite{Alvarez} can be applied.

More precisely, assuming $f(x)$ continuous,  
the expectation of the market diffusion $\E_x(f(X^{c,d}_t))$ has an ergodic limit, denoted $R(c,d)$, 
and applying the previous results, 
we can prove that the optimal controls for the player should be found in the class of reflecting controls,
considering a one variable cost function of the form $c(\cdot,R(c,d))$.
This is why we assume that the market is also controlled by reflections at some levels $c<d$,
and expect to obtain an equilibrium point when the optimal levels $a<b$ that control the player's diffusion coincide with $c<d$
(see Definition \ref{D:EquilibriumDefinition}).
Note that the question of the existence of equilibrium strategies beyond the class of reflecting controls is not addressed here.
The requirements to apply these results in the mean-field game formulation follow.
\subsection{Conditions for optimality and equilibrium}
In this setting, we can generalize the results of the section before using some simple ergodic results for diffusions.
Recall that the function $f(x)$ is assumed to be continuous.
\begin{definition}\label{D:EquilibriumDefinition}
We say that  a control $\eta^*$ is an \emph{equilibrium} of the mean-field game if it belongs to the set 
$$ 
\argmin_{\eta=(U,D) \in \mathcal{A}} \left\lbrace\limsup_{T \rightarrow \infty} \frac{1}{T} \E_x \left( \int_0^T c\big(X^\eta_s,\E_x(f(X_s^{\eta^*}))\big) ds+q_u U_T+
q_d D_T \right)\right\rbrace.
$$
In case the control is reflecting, i.e. $\eta^*=(U^{a^*},D^{b^*})$ we say that $(a^*,b^*)$ is an \emph{equilibrium point}.
\end{definition}
The idea of the above definition is to consider situations in which the individual player has no incentive to act differently to the market. 
Regarding the three-step proposal of \citet[Section 2.2]{carmona2013}, we would
(i) choose a control $\mu\in\mathcal{A}$ for the market, (ii) solve the standard stochastic problem
$$
\inf_{\eta=(U,D) \in \mathcal{A}} \left\lbrace\limsup_{T \rightarrow \infty} \frac{1}{T} \E_x \left( \int_0^T c\big(X^\eta_s,\E_x(f(X_s^{\mu}))\big) ds+q_u U_T+q_d D_T \right)\right\rbrace.
$$
to obtain a control $\eta$ (depending on $\mu$), 
and (iii) find a fixed point in $\mathcal{A}$ of the map $\mu\mapsto\eta$.  
Compared to Definition 3.2 in \cite{CDF}, 
closer to our formulation,
Definition \ref{D:EquilibriumDefinition} admits a time dependent value representing the market state. 
More precisely, in \cite{CDF}, the authors consider situations 
in which the controlled market process has a stationary distribution, 
whose mean has to coincide with the equilibrium value. 
If this is the case, as seen in Section \ref{S:theproblem}, 
the control to be an equilibrium, in general terms, should be a reflecting one.
Nevertheless, as the following results shows, when considering reflecting controls, we can substitute the time dependent value by its limit in Definition \ref{D:EquilibriumDefinition}.
%
\begin{theorem}\label{T:Limitcost}
Consider the points $ a<b$, $c<d$, and $x \in \mathbb{R}$.
Then 
\begin{multline}
\limsup_{T \rightarrow \infty} \frac{1}{T} \E_x \left( \int_0^T c(X_s^{a,b},\E_x\big(f(X_s^{c,d}))\big) ds+q_d dD^b_s+q_u dU_s^a \right) \\
=\frac{1}{m(a,b)} \left[  \int_a^b c(u,R(c,d))m(du) + \frac{q_u}{S'(a)}+\frac{q_d}{S'(b)} \right], 
\label{thm}
\end{multline}
where 
$$
R(c,d)= \int_c^d\frac{f(u)}{m(c,d)}m(du).
$$
\end{theorem} 
\begin{proof} Applying Theorem \ref{T:alvarezcost} with  the cost function $c(\cdot,R(c,d))$ 
we obtain that
\begin{multline*}
\lim_{T \rightarrow \infty} \frac{1}{T} \E_x \left(\int_0^T c(X_s^{a,b},R(c,d))ds+q_u U_T^a+q_d D^b_T\right)\\
=\frac{1}{m(a,b)} \left[\int_a^b c(u,R(c,d))m(du)
+ \frac{q_u}{S'(a)}+\frac{q_d}{S'(b)}\right],
\end{multline*}
i.e. the r.h.s. in \eqref{thm}.
It remains then to verify that
\begin{equation}\label{E:limitcostiszero}
\limsup_{T \rightarrow \infty} \frac{1}{T} \E_x \left( \int_0^T \vert c(X_s^{a,b},\E_x(f(X_s^{c,d})))-c(X_s^{a,b},R(c,d))\vert   ds\right) =0.
\end{equation}
In order to do this, define the continuous function $H\colon f([c,d]) \rightarrow \mathbb{R}^{+}$ by
\[
H(y)= \max_{u \in [a,b]}  \vert c(u,y)-c(u,R(c,d)) \vert,
\]
and observe that the limit in \eqref{E:limitcostiszero} can be bounded by
\begin{multline*}
\limsup_{T \rightarrow \infty} \frac{1}{T}\int_0^T  H(\E_x\big(f(X_s^{{c},{d}})))ds \\
=\limsup_{T \rightarrow \infty} \frac{1}{T}  \int_0^T    H \left( \int_c^d f(y) \P_s (x,dy)\right)ds,  \end{multline*}
with $\P_s(x, dy)= \P_x(Y^{{c},{d}}_s \in dy)$. 
This limit is zero because 
$$
H \Bigg( \int_c^d f(y) \P_s (x,dy)\Bigg)\to H(R({c},{d}))=0, 
$$  
as  $H$ is uniformly continuous, bounded and  
$$ 
\left\Vert  \P_s(x, \cdot)- \frac{1}{m({c},{d})}m(\cdot)  \right\Vert \rightarrow 0,\qquad\text{as $s \to \infty$},
$$
with the norm of total variation (see Theorem 54.5 in \cite{RW}). 
It follows that \eqref{E:limitcostiszero} holds, concluding the proof.
\end{proof}
The existence and uniqueness of minimizers given in (b) in Theorem \ref{T:alvarezcost} can also be generalized, 
by noticing that in Theorem \ref{T:Limitcost} the second variable in the cost function is fixed. 
The optimality of reflecting controls within the class of càdlàg controls corresponding to Definition \ref{D:EquilibriumDefinition}
follows from Theorem \ref{T:Alvarezcompletation}.
\begin{theorem}\label{T:optimal}
For a fixed $({a},{b})$, the infimum of the ergodic problem is reached only at a pair $(a^{\ast},b^{\ast})$ such that 
\begin{itemize}
\item[\rm(i)]
$\pi_1(b^{\ast},R({a},{b}))= \pi_2(a^{\ast},R({a},{b})),$
\item[\rm(ii)]
$ \displaystyle\int_{a^{\ast}}^{b^{\ast}} \left( \pi_1(t,R({a},{b}))-\pi_1(b^{\ast},R({a},{b})) \right) m(dt) + \frac{\displaystyle q_u+q_d}{\displaystyle S'(a^{\ast})}=0. $
\end{itemize}
Moreover $(a^{\ast},b^{\ast}) \in (-\infty,x_2^{R(a,b)}) \times (x_1^{R(a,b)},\infty)$
\end{theorem} 
Based on this result we obtain a condition for equilibrium of the mean-field game (see Definition \ref{D:EquilibriumDefinition}).
\begin{theorem}\label{L:Equilibriumpoint}
A pair $a<b$ is an equilibrium point if and only if
\begin{itemize}
\item[\rm(i)]
$\pi_1(b,R(a,b))= \pi_2(a,R(a,b)),$
\item[\rm(ii)]
$ \displaystyle\int_{a}^b \left[ \pi_1(t,R(a,b))-\pi_1(b,R(a,b)) \right] m(dt) + \frac{\displaystyle q_u+q_d}{\displaystyle S'(a)}=0.$
\end{itemize}
Moreover $(a,b) \in (-\infty,x_2^{R(a,b)}) \times (x_1^{R(a,b)},\infty)$
\end{theorem}
\subsection{The multiplicative case}\label{Subsectionmultplicative}
In this subsection, we assume that the cost function has a multiplicative form. 
\begin{assumptions}\label{assumption35}
The cost function  satisfying Assumption \ref{A:Assumptions2}, is factorized as
\[ 
c(x,y)= g(x)h(y),
\]
where the factors satisfy
\begin{itemize}
\item[\rm(i)] $g\colon\R\to[0,\infty)$ is a convex function, with $g(x)\geq g(0)$, 
\item[\rm(ii)] $h\colon\R\to(0,\infty)$ is continuous, with $h(x)\geq h(0)$.
\end{itemize}
\end{assumptions}
Note that such a multiplicative decomposition is particularly natural when $g(x)$ is interpreted as a standardized representation of the units of a good corresponding to a state $x$ and $h(y)$ as the factor modeling the unit cost based in the market.

We give a first result that follows from Theorem \ref{L:Equilibriumpoint} if the cost function is multiplicative. 
In this situation, using condition (i), one of the variables can be obtained as a function of the other.
For this purpose, consider the set
\begin{equation*}
C_a = \lbrace b \in \mathbb{R} \colon b > x_1^{R(a,b)} \vee a, \ x_2^{R(a,b)} > a, \ \pi_1 (b,R(a,b))=\pi_2(a,R(a,b)) \rbrace .  
\end{equation*}
Observe that if $C_a=\emptyset$, there are no equilibrium points.
We then assume condition $C_a \neq \emptyset$ if and only if $a \leq 0$. 
This means that we search for the equilibrium points in a connected set. 
Furthermore, for a fixed $a\leq 0$ we denote 
\begin{equation}
\label{E:firstconditionforequilibrium}
\rho(a)=\inf C_a,
\end{equation}
and
$$
L(a)=R(a,\rho(a)).
$$
\begin{proposition}\label{T:Equilibrium}
Suppose that the cost function factorizes as in Assumption \ref{assumption35},
and there exists a point $a_0 \leq 0$ such that the function $\rho $ defined via \eqref{E:firstconditionforequilibrium} is continuous in $(-\infty, a_0]$. 
Then,
\begin{itemize}
\item[$\mathrm(C_1)$]  if
$$ \int_{a_0}^{\rho(a_0)} (\pi_1 (t,L(a_0))-\pi_1(\rho(a_0),L(a_0)))m(dt) +\frac{q_u+q_d}{S'(a_0)} \geq   0, $$
then there is at least one equilibrium point. 
\item[$\mathrm(C_2)$] Furthermore, if  in $(-\infty, a_0]$, 
\begin{multline*}
     \pi_2 (t,L(a_2))- \pi_2(a_2,L(a_2)) < \pi_2 (t,L(a_1))- \pi_2(a_1,L(a_1))  \nonumber   \\  \forall (a_2,a_1,t) \ \text{ s.t, } a_2<a_1<t \leq a_0,  \end{multline*}
\begin{multline*}
    \pi_1(t,L(a_2) )- \pi_1(\rho(a_2),L(a_2)) < \pi_1(t,L(a_1) )- \pi_1(\rho(a_1),L(a_1)) \nonumber \\ \forall (a_2,a_1,t) \ \text{ s.t. }  \rho(a_2) >\rho(a_1) > t  \geq a_0, 
    \end{multline*}
and
\begin{align}\label{eq:c2condition2}
 \int_r^{l} (\pi_1 (t,R(r,l))-\pi_1(l,R(r,l)))m(dt)+ \frac{q_u+q_d}{S'(r)} >0, \nonumber \hspace{15mm} \\  \forall r \in (a_0,\rho(a_0)), l >r , \  \pi_1 (l,R(r,l))=\pi_2(r,R(r,l)),
\end{align}
then the equilibrium is unique.
\end{itemize}
\end{proposition}
\begin{proof}
For the existence of equilibrium points, we need to prove 
\begin{equation*}
    \int_{A}^{\rho(A)} (\pi_1 (t,L(A))-\pi_1(\rho(A),L(A)))m(dt) +\frac{q_u+q_d}{S'(A)} <  0, 
\end{equation*}
for some $A<a_0$. First, observe that the inequality can be rewritten as
\begin{multline}
  \int_{A}^{0} (\pi_2 (t,L(A))-\pi_2(A,L(A)))m(dt) \\
  +\int_{0}^{\rho(A)} (\pi_1 (t,L(A))-\pi_1(\rho(A),L(A)))m(dt) +\frac{q_u+q_d}{S'(0)}<0.    
\end{multline}
    Furthermore,
    due to the nature of the multiplicative cost, the points $x_i^y,i=1,2$ defined in \eqref{A:Assumptions2} can be taken all equal to  $x_i^0$ for each $i$ respectively. Thus, for $A$
 negative enough, both integrands are always negative and tend to $-\infty$ when $A \rightarrow -\infty$.  
Finally, for the uniqueness, condition $(C_2)$ implies that the map defined in $(-\infty,a_0]$:
$$
a  \rightarrow \int_{a}^{\rho(a)} (\pi_1 (t,L(a))-\pi_1(\rho(a),L(a)))m(dt) +\frac{q_u+q_d}{S'(a)}, 
$$ 
is monotone, thus concluding that the root of this map is unique.
\end{proof}
\begin{remark}
    Condition $\mathrm(C_2)$ is a condition on differences of value functions. 
    In particular,
    if we assume $\pi_2 \in C^2((-\infty,a_0) \times  \mathbb{R}) $, $f$ defined in the introduction of the section is increasing and $L(a)$ is increasing, then the first inequality in condition $\mathrm(C_2)$ holds if   $\pi_2$ has negative cross second derivative in $(-\infty, a_0) \times \mathbb{R}$ which is equivalent to the function 
    $$(a, \mu) \rightarrow \pi_2(a, \langle f,\mu\rangle), \quad  a \in (-\infty,a_0), \ \mu \text{ a  probability measure,} $$
    being submodular (see Example $2$ of
    \cite[Assumption 2.9]{DFFN}). 
    A similar analysis can be made with the second inequality (the function in this case is supermodular).
 \end{remark}
In the particular case of a diffusion without drift, the conditions of the previous proposition are satisfied under the following simple conditions.
\begin{corollary}
Suppose that the cost function factorizes as in Assumption \ref{assumption35}.
Assume furthermore that
$g$ is unbounded, convex and with minimum at zero, and the diffusion process \eqref{sde} has no drift. 
Then,
\vskip1mm\par\noindent
{\rm(a)} the function $\rho(a)$ is defined as the unique solution of the equation 
$h(a)=h(b)$, with $a \leq 0 \leq b$,
and there exists an equilibrium point,
\vskip1mm\par\noindent
{\rm(b)}
if the function
$h(R(a,\rho(a)))$
is strictly decreasing for $a\leq 0$, the equilibrium is unique. 
\end{corollary}
\begin{proof}
Take  $a_0= 0 $.  
We have that $\pi_1(b,R(a,b))=\pi_2 (a,R(a,b))$ is equivalent to the equality $g(b)=g(a)$, 
thus from the fact that $g$ is convex with a minimum at zero, 
the restriction of $g$ to $x<0$ is an invertible function, 
denote it by $g_{\vert_{(-\infty,0)}}$, 
and we can define
$$
\rho(a)=\left(g_{\vert_{(-\infty,0)}}\right)^{-1}(a).
$$  
We conclude part (a)  from the fact $\rho(0)=0$ and condition $(C_1)$ and is fulfilled.
Condition $(C_2)$ is verified, the first two statements follow from the monotonicity of $h$ and $a \rightarrow g(a,R(a))$ because the inequalites can be rewritten as: 
\begin{multline*}
     (g(t)-g(a_2))h(R(a_2,\rho(a_2))) < (g(t)-g(a_1))h(R(a_1,\rho(a_1))) \nonumber   \\  \forall (a_2,a_1,t) \ \text{ s.t, } a_2<a_1<t \leq 0,  
     \end{multline*}
\begin{multline*}
    (g(t) - g(\rho(a_2)) )h(R((a_2),\rho(a_2))) < (g(t)-g(a_1))h(R(a_1,\rho(a_1))) \nonumber \\ \forall (a_2,a_1,t) \ \text{ s.t. }  \rho(a_2) >\rho(a_1) > t  \geq 0. 
    \end{multline*}
The third integral \eqref{eq:c2condition2} condition in $(C_2)$ is automatic, as $(a_0,\rho(a_0))=(0,0)$.
\end{proof}
\section{Examples}\label{S:examples}
We present below several examples where the equations of Theorem \ref{L:Equilibriumpoint} can be expressed more explicitly and solved numerically.
To help the presentation, for each example, we plot in an $(a,b)$ plane the implicit curves defined by these equations. 
To this end, we write equation (i) in Theorem \ref{L:Equilibriumpoint} as
$$
F(a,b)=\pi_1(a,R(a,b))-\pi_2(b,R(a,b))=0,
$$
and draw first the set of its solutions. 
We then draw the set determined by condition \rm{(ii)}. 
Note that there are cases where there is an intersection of both curves outside the set 
$\lbrace a < b \rbrace$, these points are of no interest for our problem.
In all examples the function affecting the market expectation is $f(x)=x$.
Furthermore, to ease of exposition, we present the conclusions and the plots and defer the computations to the Appendix (see Subsection \ref{calc}).
\subsection{Examples with multiplicative cost}
The cost function now has the form
\begin{equation}\label{E:cost}
    c(x,y)= \max (-\lambda x,x) (1+ \vert y \vert^{\beta}),\quad  \lambda >0,\quad \beta \geq 1, 
\end{equation}
and $q_d \lambda = q_u$.
\begin{remark}
In this scenario the value $\max (-\lambda x, x)$ could represent the  maintenance cost of certain property done by a third party. This third party will change the price of its services depending on the demand of the market. 
\end{remark}
We consider a mean reverting process $X=\{X_t\}$ that follows the stochastic differential equation
\begin{equation}\label{ou}
dX_t= -\theta X_t dt+ \sigma(X_t) dW_t, 
\end{equation}
such that $\sigma$ is a function that satisfies the conditions of Section \ref{S:theproblem} and $q_d \theta <1$. Under these conditions the function $c(x,y)$ is under Assumptions \ref{A:Assumptions2}. First observe that if we take $x^y =0$ for all $y \in \mathbb{R}$, then $c(x,y ) \geq c(x^y,y)=0$, Second, by taking $K_y=0, \ \alpha_y= \lambda \wedge 1$ for all $y \in \mathbb{R}$,  condition \eqref{E:kalpha} is satisfied. Finally observe that for every $y \in \mathbb{R}$ the maps   $
\pi_1(x,y), \ \pi_2(x,y) $ 
are decreasing on $x$ in $(-\infty,0)$, increasing on $x$ in $(0,\infty)$ and both conditions $\rm{(i)}$ and $\rm{(ii)}$ in Assumptions \ref{A:Assumptions2} are satisfied.

In the particular case when $\sigma$ is constant, we can compute
\[ 
R(a,b)= \sqrt{\frac{\sigma^2 }{ \theta \pi}} \left(  \frac{\displaystyle e^{-a^2 \frac{\theta}{\sigma^2}}  -e^{-b^2 \frac{\theta}{\sigma^2}}   }{\mathrm{erf}\left(\sqrt{\frac{\theta}{\sigma^2}} b\right)    - \mathrm{ erf} \left( \sqrt{\frac{\theta}{\sigma^2}} a \right) }                        \right), 
\]
where $\mathrm{erf}(x)=\frac1{\sqrt{2\pi}}\int_{-\infty}^xe^{-y^2/2}\,dy$.
Using Proposition \ref{T:Equilibrium}, existence of equilibrium points holds.
Furthermore, if $\sigma$ is even then uniqueness also holds. 
Again, the calculations are in Appendix \ref{S:SubSubsectionOUMC}. 
In the graphical examples below $\sigma$ is constant. 
\begin{figure}[H]
\centering
\begin{subfigure}{.5\textwidth}
  \centering
  \includegraphics[width=.8\linewidth]{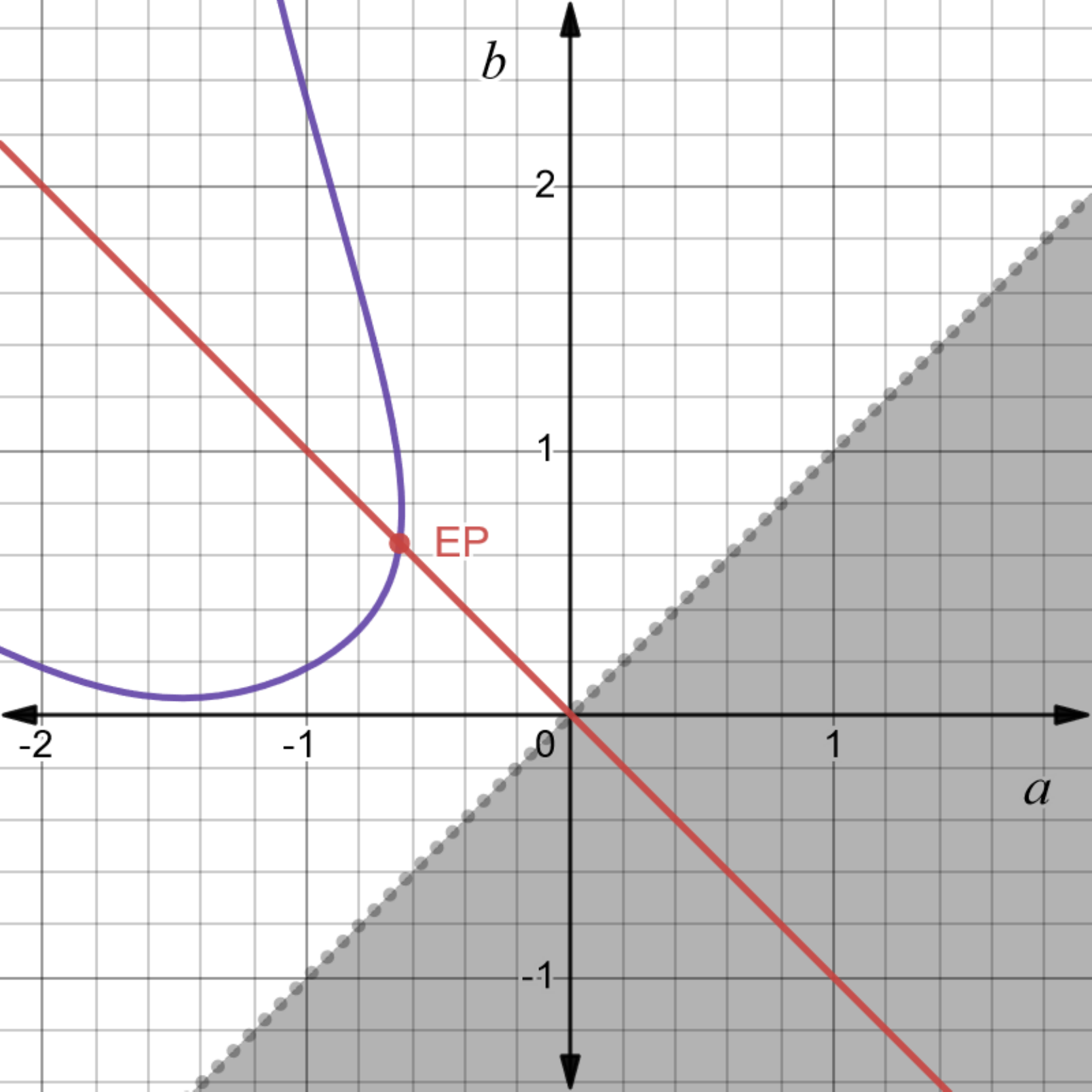}
  \label{fig:sub1}
\end{subfigure}%
\caption{Mean reverting process \eqref{ou} with multiplicative cost and parameters 
$\theta=0.4,q_d=0.1, \lambda=1, \sigma=2, \beta=1$. 
The equilibrium point (EP) is $(-0.646,0.646)$ with value $0.617$.}
\end{figure}
\subsection{``Follow the market" examples}
The idea is to introduce a cost function in such a way that the player has incentives to follow the market evolution.
The cost function is then
$$
c(x,y)= \vert   x - y \vert. 
$$
\subsubsection{Brownian motion with negative drift }\label{S:BrownianmotionwithdriftPUNISHMENT}
In this case, the driving process $X=\{X_t\}$ is
$$
X_t=\mu t+W_t,
$$
where $\mu<0$.
 We proceed to prove  that Assumption \ref{A:Assumptions2} is satisfied. By taking $x^y =y$ for all $y \in \mathbb{R}$, then $c(x,y ) \geq c(x^y,y)=0$, Second, by taking $K_y=\vert y \vert, \ \alpha_y= 1$ for all $y \in \mathbb{R}$ then \eqref{E:kalpha} is satisfied. Finally observe that for every $y \in \mathbb{R}$ the maps $ 
\pi_1(x,y) , \ \pi_2(x,y)
$ 
are decreasing on $x$ in $(-\infty,y)$, increasing on $x$ in $(y,\infty)$ and both conditions $\rm{(i)}$ and $\rm{(ii)}$ in Assumptions \ref{A:Assumptions2} are satisfied.\\
The problem can be reduced to a one variable problem. The conclusions are:
\begin{itemize}
    \item If there is a positive constant $C$ such that
  \begin{align*}
 C(1+e^{2 \mu C})(1-e^{2 \mu C})^{-1}+(q_u+q_d)\mu + \mu^{-1}&=0,\\ 
 \Big(\frac{\displaystyle C}{\displaystyle e^{2 \mu C}-1 } \Big) \frac{\displaystyle 2 e^{2 \mu C}}{\displaystyle \mu} + \frac{\displaystyle -2e^{2 \mu C}+2C \mu +1}{\displaystyle 2 \mu^2} +q_d +q_u&=0,
\end{align*}
    then every point of the set $\lbrace (a,a+C), a \in \mathbb{R} \rbrace$ is an equilibrium point.
    \item Otherwise there are no equilibrium points.
\end{itemize}
The details can be found in the Appendix \ref{S:SubSubsectionBMWD} 
\begin{figure}[!htb]
\centering
\begin{subfigure}{.5\textwidth}
  \centering
\includegraphics[width=5cm]{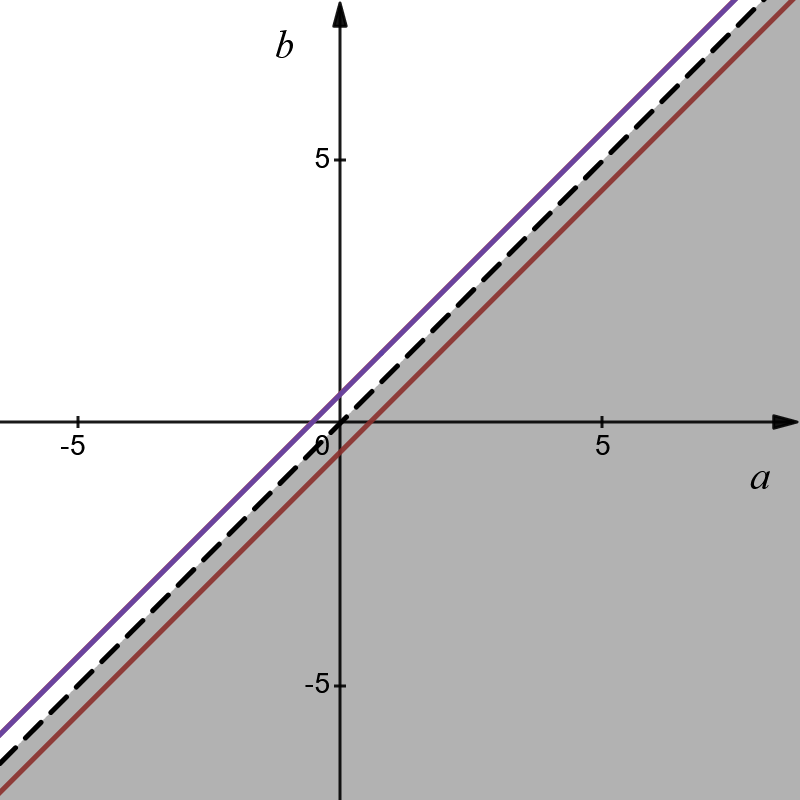}
\end{subfigure}%
\begin{subfigure}{.5\textwidth}
  \centering
  \includegraphics[width=.8\linewidth]{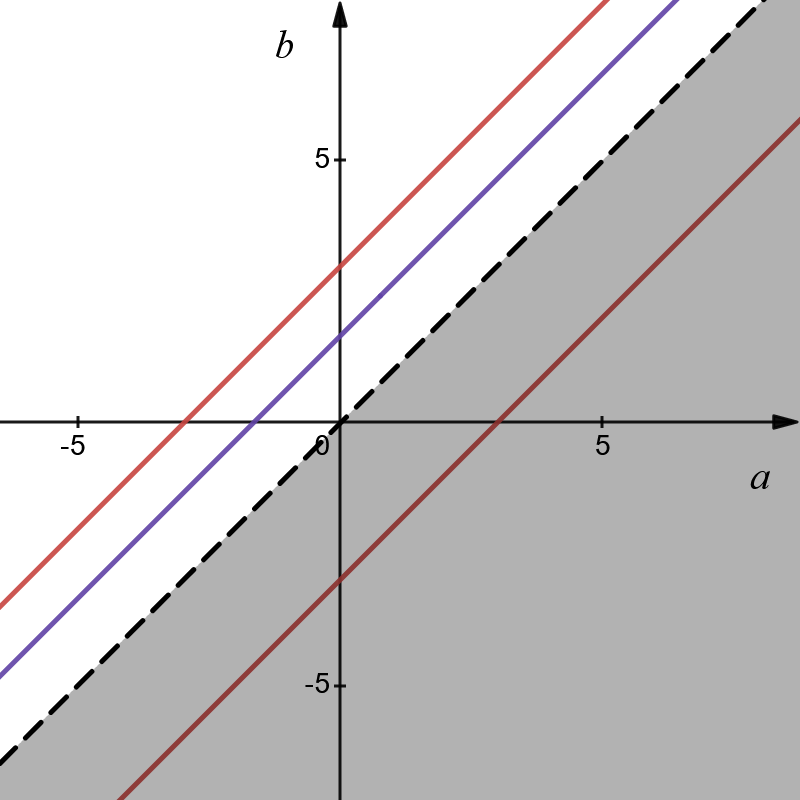}
\end{subfigure}
\caption{Brownian motion with drift and cost function $c(x,y)=|x-y|$. 
On the left ($q_u+q_d=0.1, \mu=-0.89
$) the value at equilibrium points is constant $0.848$. 
On the right ($q_u+q_d=2, \mu=-1$ ) there are no equilibrium points}
\end{figure}
\subsubsection{Ornstein Uhlenbeck process}
In this case, the process $X=\{X_t\}$ follows the stochastic differential equation
\begin{equation*}
dX_t= -\theta X_t dt+ \sigma dW_t, 
\end{equation*}
We analyze the symmetric case when $q:=q_d=q_u$ and $q \theta <1$. 
In this situation, taking the same parameters as in the previous example, 
$c(x,y)$ is under Assumption \ref{A:Assumptions2}.
The existence of equilibrium points will hold, but uniqueness not necessarily. 
Essentially, the equation $\pi_1(a,R(a,b))=\pi_2(a,R(a,b))$ is satisfied when $a=-b$ by symmetry,
so similar arguments as the ones in the multiplicative case hold. 
However the line $a+b=0$ is not the only set where $\pi_1(a,R(a,b))=\pi_2(a,R(a,b))$. 
We show that uniqueness does not always hold, see Figure \ref{fig:3}.
\begin{figure}
\centering
\includegraphics[width=5.5cm]{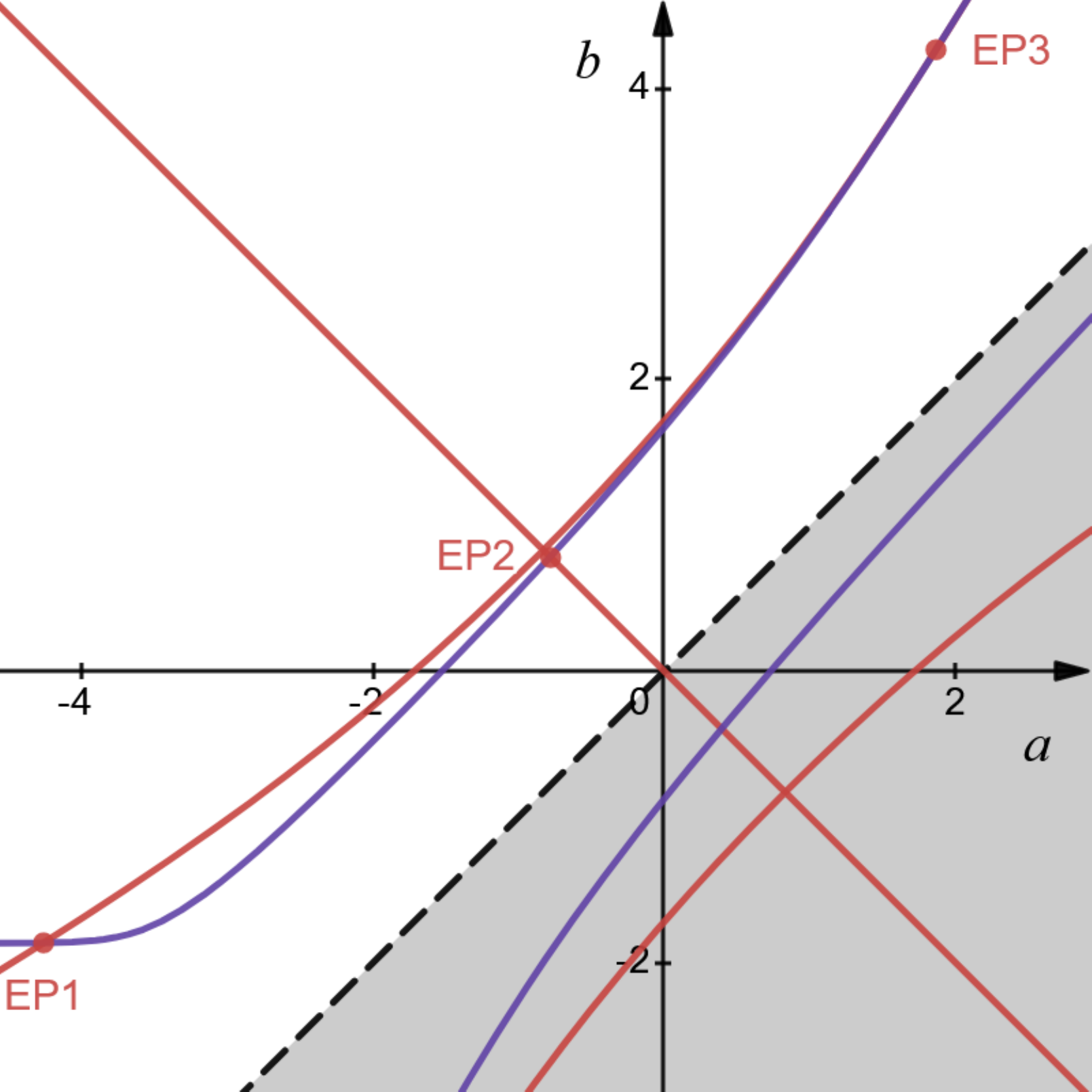}
\caption{Mean reverting process with $q=0.1$, $\theta=3$, $s=2$, $EP1\sim (-4.26,-1.86)$, $EP2\sim (-0.78,0.78)$, $EP3 \sim (1.87,4.27)$ with the values $0.839$, $0.55$ and $0.84$ at each equilibrium point  respectively}\label{fig:3}
\end{figure}
\section{Approximation of Nash equilibria in symmetric $N$-player games with mean-field interaction}\label{S:epsilon}
In this section, 
we present an approximation result for Nash equilibria in the $N$-player
game corresponding to the ergodic mean-field game considered above, when the number of players $N$
tends to infinity.
More precisely, we establish that an equilibrium point of the mean-field game 
of Definition \ref{D:EquilibriumDefinition}
 is an $\epsilon$-Nash equilibrium of the corresponding $N$-player game of Definition \ref{D:epsilonNashequilibrium}, for $N$ large enough.
These approximation results have been studied for instance in
\cite{CG} and \cite{CDF} and the references therein.
In order to
formulate the approximation result, consider:
\begin{itemize}
\item[\rm(i)] A filtered probability space 
$(\Omega,\mathcal{F},\lbrace \mathcal{F}_t\colon t\geq 0^-\rbrace,\P)$ 
that satisfies the usual conditions, where all the processes are defined.
\item[\rm(ii)] 
Diffusion processes $X,\{X^i\}_{i=1,2,\dots}$,
each of one satisfies equation \eqref{sde} driven by respective adapted independent Brownian motions $W,\{W^i\}_{i=1,2,\dots}$.
\item[\rm(iii)] The set of admissible controls $\mathcal A$ of Definition \ref{D:admissiblecontrols}, that in particular assumes, given an admissible control $\eta^i=(U^i,D^i)$, the existence of the controlled process as a solution of
\begin{equation}\label{sdeci}
dX^{i,\eta^i}_t= \mu(X^{i,\eta^i}_t)dt + \sigma (X^{i,\eta^i}_t)dW^i_t+dU^i_t-dD^i_t, \quad X^i_{0^-}=x^i,
\end{equation}
for each $i=1,2,\dots$ 
\end{itemize}
For simplicity and coherence  we denote by $X^{i,a,b}$ the solution to \eqref{sdeci} when the $i$-th player chooses 
reflecting strategies within $a<b$, denoted respectively by $U^{i,a}$ and $D^{i,b}$.
As usual, we define a vector of admissible controls by
$$
\Lambda=(\eta^1, \dots ,\eta^N)
$$ 
such that $\eta^i=(U^i,D^i)$ is an admissible control selected by the player $i$ 
in the $N$-player game.
Furthermore, we define 
$$
\aligned
\Lambda^{-i}&=(\eta^1, \dots, \eta^{i-1}, \eta^{i+1}, \dots ,\eta^N),\\
(\mu,\Lambda^{-i})&=(\eta^1, \dots, \eta^{i-1},\mu,\eta^{i+1}, \dots ,\eta^N)
\endaligned
$$ 
and, given a real continuous function $f(x)$, denote
\begin{equation}\label{E:nash}
\bar{f}^{-i}_s=\frac{1}{N-1}  \sum_{j \neq i}^N f( X_s^{j,\eta^j}),
\quad
\bar{f}^{a,b,-i}_s=\frac{1}{N-1}  \sum_{j \neq i}^N f( X_s^{j,a,b}),
\end{equation}
and, given $\mu=(U,D)\in\mathcal{A}$,
for $(\mu, \Lambda^{-i})$, consider
\begin{equation}\label{eq:vin}
V_N^{i}(\mu, \Lambda^{-i})(x)= 
\limsup_{T \rightarrow \infty} \frac{1}{T} \E_x 
\Bigg( \int_0^T c \left(X_s^{i,\mu},\bar{f}^{-i}_s\right)ds
+q_u U^i_T+q_d D^i_T \Bigg), 
\end{equation}
for a cost function  $c(x,y)$ satisfying Assumption \ref{A:Assumptions2}. 
\begin{definition}\label{D:epsilonNashequilibrium}
For fixed $\epsilon>0$ and $N\in\mathbb{N}$,  a vector of admissible controls
$\Lambda=(\eta^1,\dots ,\eta^N)$  is called an $\epsilon$-Nash equilibrium if 
for all $i$ and all $x\in\R$,
$$
V_N^{i}(\eta^i, \Lambda^{-i})(x) \leq V^i_N(\mu,\Lambda^{-i} )(x)+\epsilon,\quad\text{for all $\mu\in\mathcal{A}$}. 
$$
\end{definition}
We are ready to prove that the equilibrium points of the mean-field game are 
$\epsilon$-Nash equilibriums for the $N$-player game in two different situations: 
(\rm{i}) with reflecting controls for the players and a cost function that is convex in the second variable, 
(\rm{ii}) with general controls in $\mathcal{A}$, and the cost function $c(x,y)=|x-y|$. 
\begin{theorem}\label{T:GREATAPROXIMATION}
Consider a cost function $c(x,y)$ that satisfies Assumption \ref{A:Assumptions2}, 
and suppose that the function $f(x)$ in Def. \ref{D:EquilibriumDefinition} is continuous.
Assume also that one of the following conditions holds:
\begin{itemize}
    \item[\rm(i)]
For every fixed $x$ the function $y\mapsto c(x,y)$ is convex, 
and the set of admissible controls for each process $X^i, \ i=1, \dots ,N$, 
is the set of reflecting controls instead of $ \mathcal{A}$.
\item[\rm(ii)] 
We have $f(x)=x$ and the cost function is $c(x,y)=\vert x-y \vert$.
\end{itemize}
Then, if $(a,b)$ is an equilibrium point for the mean field game driven by $X$,
given $\epsilon>0$,
the vector of controls 
\begin{equation}\label{E:equal}
    \Lambda^{a,b}=((U^{1,a},D^{1,b}),\dots,(U^{N,a},U^{N,b})),
\end{equation}
is an $\epsilon$-Nash equilibrium for the $N$-player game,
for $N$ large enough.
\end{theorem}
In the proof of (i) we will use the following result.
\begin{lemma}\label{lemma:convexity}
Let $c(x,y)$ be a positive measurable function such that $y\mapsto c(x,y)$ is convex for each fixed $x$, and $(X,Y)$ a random vector. Then
\vskip1mm\par\noindent
{\rm(a)} If $X$ and $Y$ are independent,
\begin{equation}\label{eq:statement}
\E c(X,\E Y)\leq \E c(X,Y).
\end{equation}
\vskip1mm\par\noindent
{\rm(b)} In the general case, statement \eqref{eq:statement} is not true.
\end{lemma}
\begin{proof}[Proof of Lemma \ref{lemma:convexity}] (a) With $F_X$ and $F_Y$ the respective distributions of $X$ and $Y$, we have
\begin{align*}
\E c(X,Y)	&=\int\left[\int c(x,y)F_Y(dy)\right]F_X(dx)\\
		&\geq\int c\left(x,\int yF_Y(dy)\right)F_X(dx)=\E c(X,\E Y).
\end{align*}
To see (b), consider $c(x,y)=|x-y|$, a standard normal random variable $X\sim\mathcal{N}(0,1)$, and the random vector
$(X,Y)=(X,X)$. We have
$$
\E c(X,Y)=\E|X-X|=0<\sqrt{2\over\pi}=\E c(X,\E Y)=\E|X|,
$$
giving the counter-example that concludes the proof.
\end{proof}
\begin{proof}[Proof of {\rm(i)} in Theorem \ref{T:GREATAPROXIMATION}]
Define the function 
\begin{equation}\label{Eq:Vcost}
    V\colon\mathcal{A} \times \lbrace (a,b)\colon a<b \rbrace \rightarrow \mathbb{R}
\end{equation}
by the formula
\begin{equation*}
V(\mu,(a,b))=  
\limsup_{T \rightarrow \infty} \frac{1}{T} \E_x \left(\int_0^T c\big(X_s^\mu,\E_x(f(X^{a,b}_s))\big) ds +q_u U_T+q_d D_T \right),
\end{equation*}
where $\mu=(U,D)$. 
Take $\Lambda^{a,b}$ as in \eqref{E:equal}.
The departing point is the inequality provided by the equilibrium definition:
\begin{equation}\label{E:departing}
V((U^{a},D^{b}),(a,b))\leq V(\mu,(a,b)),\quad\text{for any $\mu\in\mathcal{A}$}.
\end{equation}
Second, by equidistribution of the player's driving processes,
$$
\E_xc(X^\mu_s,\E_x(f(X^{a,b}_s)))=\E_xc(X^\mu_s,\E_x(\bar f^{a,b,-i}_s)).
$$
Now, taking $c<d$ and $\mu=(U^c,D^d)$, by convexity and independence between the coordinates, we apply (i) in Lemma \ref{lemma:convexity}:
$$
\E_xc(X^{c,d}_s,\E_x(\bar f^{a,b,-i}_s))\leq \E_xc(X^{c,d}_s,\bar f^{a,b,-i}_s),
$$
Integrating in time, taking expectation and ergodic limits, combined with \eqref{E:departing},
it follows
\begin{equation}\label{E:following}
V((U^{a},D^{b}),(a,b))\leq  V((U^{c},D^{d}),(a,b))\leq V^i_N((U^c,D^d),\Lambda^{a,b,-i}_N).
\end{equation}
Now, as $f(x)$ is continuous, the set $f([a,b])$ is a closed interval, denote it by $[m,M]$,
and  observe that 
$$
(X^{i,a,b}_s,\bar f^{a,b,-i}_s)\in [a,b]\times[m,M], 
$$
that is a product of closed intervals. 
Then, as $c(x,y)$ is uniformly continuous in this compact domain, 
given $\epsilon$ there exist $\delta$ s.t.
$$    |c(X^\mu_s,\bar f^{a,b,-i}_s)-c(X^\mu_s,\E_x(f(X^{a,b}_s))|\leq \frac\epsilon2,
$$
whenever $|\bar f^{a,b,-i}_s-\E_x(f(X^{a,b}_s))|\leq\delta$.
Now we apply Hoeffding's inequality for bounded random variables $m\leq f(X^{j,a,b})\leq M$,
obtaining,
$$
\P\left(|f^{a,b,-i}-\E_x(f(X^{a,b}_s))|\geq\delta\right)\leq 2e^{-{2\delta^2(N-1)\over (M-m)^2}}.
$$
Finally, denoting $\|c\|_\infty=\max\{|c(x,y)|\colon a\leq x\leq b, m\leq y\leq M\}$, we have
\begin{multline*}
\left|\frac1T\E_x\int_0^T\left(c(X^{i,a,b}_s,\bar f^{a,b,-i}_s)-c(X^{i,a,b}_s,\E_x(f(X^{a,b}_s))\right)\,ds\right|\\
\leq \frac\epsilon2+\frac{2\|c\|_\infty}T\int_0^T\P_x\left(|\bar f^{a,b,-i}_s-\E_x(f(X^{a,b}_s))|\geq \delta\right)\,ds\\
\leq \frac\epsilon2+4\|c\|_\infty e^{-{2\delta^2(N-1)\over (M-m)^2}}\leq\epsilon,
\end{multline*}
for $N$ large enough.  From this follows that, for these values of $N$, 
$$
\left|V((U^{a},D^{b}),(a,b))-V^i_N((U^{a},D^{b}),\Lambda^{a,b,-i}_N)\right|\leq \epsilon,
$$
concluding, in view of \eqref{E:following}, the proof of (\rm{i}).
\end{proof}
\begin{proof}[Proof of {\rm(ii)} in  Theorem \ref{T:GREATAPROXIMATION}] 
As $f(x)=x$, we denote
\begin{equation*}
\bar{X}^{a,b,-i}_{s,N}=\frac{1}{N-1}  \sum_{j \neq i}^N X_s^{j,a,b}.
\end{equation*}
As $(a,b)$ is an equilibrium point of the mean field game, 
given $\epsilon>0$,
we have to prove that
\begin{equation}\label{eq:nash}
V_N^{i}((U^{i,a},D^{i,b}), \Lambda^{a,b,-i})\leq V_N^{i}(\mu,\Lambda^{a,b,-i})+\epsilon,
\end{equation}
for any strategy $\mu\in\mathcal{A}$, for $N$ large enough.
Observe now that, given a strategy $\eta$, if for some $N_0$ and some $i_0$, 
 we have $V_{N_0}^{i_0}(\eta,\Lambda^{a,b,-i_0})<\infty$, then
\begin{align*}
\limsup_{T \rightarrow \infty} \frac{1}{T} \E_x 
\int_0^T\left|
X_s^\eta-\bar{X}^{a,b,-i_0}_{s,N_0}
\right|
ds=:I_0<\infty,\\
\limsup_{T \rightarrow \infty} \frac{1}{T} \E_x (q_u U_T)=:J_0<\infty,\\
\limsup_{T \rightarrow \infty} \frac{1}{T} \E_x (q_d D_T)=:K_0<\infty.
\end{align*}
By adding and substracting $\bar{X}^{a,b,-i_0}_{s,N_0}$ and the triangular inequality, it follows
$$
\limsup_{T \rightarrow \infty}
\frac{1}{T} \E_x 
\int_0^T\left|
X_s^\eta
\right|
ds\leq I_0+\max(|a|,|b|),
$$
and in consequence
$$
\max(V(\eta,(a,b)),V_N^{i}(\eta,\Lambda^{a,b,-i}))\leq I_0+J_0+K_0+2\max(|a|,|b|),
$$
for all $N$ and $i$.
Then, in order to prove \eqref{eq:nash}, it is enough to  consider these strategies $\eta$.
Now, as $(a,b)$ is an equilibrium point,
we have
\begin{multline*}
V_N^{i}((U^{i,a},D^{i,b}),\Lambda^{a,b,-i})- V_N^{i}(\eta,\Lambda^{a,b,-i})\\
=V_N^{i}((U^{i,a},D^{i,b}),\Lambda^{a,b,-i})- V(\eta,(a,b))
+V(\eta,(a,b))- V_N^{i}(\eta,\Lambda^{a,b,-i})\\
\leq
V_N^{i}((U^{i,a},D^{i,b}),\Lambda^{a,b,-i})- V((U^a,D^b),(a,b))\\
+V(\eta,(a,b))- V_N^{i}(\eta,\Lambda^{a,b,-i})\\
\leq
2\sup_{\eta}\left|
V(\eta,(a,b))- V_N^{i}(\eta,\Lambda^{a,b,-i})
\right|.
\end{multline*}
By the triangular inequality, given $\eta$, we have
\begin{multline*}
\left|V(\eta,(a,b))- V_N^{i}(\eta,\Lambda^{a,b,-i})\right|\\
\leq
\limsup_{T \rightarrow \infty} \frac{1}{T} \E_x 
\int_0^T\left|
|X_s^\eta-\bar{X}^{a,b,-i}_s|-|X_s^\eta-\E_x(X_s^{a,b})|
\right|
ds\\
\leq
\limsup_{T\rightarrow \infty} \frac{1}{T}  
\int_0^T\E_x|\bar{X}^{a,b,-i}_s-\E_x(X_s^{a,b})|
ds\leq {b-a\over\sqrt{N-1}},
\end{multline*}
because
$$
\E_x|\bar{X}^{a,b,-i}_s-\E_x(X_s^{a,b})|\leq
\sqrt{{1\over N-1}\var_x(X_s^{a,b})}\leq{b-a\over \sqrt{N-1}},
$$
concluding the proof.
\end{proof}
%
\subsection*{Statements and Declarations}
None of the authors have any conflicts of interest.
The second and third authors are supported by CSIC - Proyecto Grupos nr. 22620220100043UD, Universidad de la Rep\'ublica, Uruguay.

%
\appendix
\section{Appendix}\label{A}
\subsection{Calculations of the examples}\label{calc}
\subsubsection{Absolute value, Brownian motion with negative drift}\label{S:SubSubsectionBMWD}
In this case (see \cite{BS}), 
$$
S'(x)=\exp (-2\mu x), \ m'(x)=2e^{2\mu x}. 
$$
Therefore
\begin{equation*}
R(a,b)
=\frac{\displaystyle \int_a^b 2ue^{2\mu u} du}{\displaystyle  \int_a^b 2e^{2\mu u} du} 
= \frac{\displaystyle  b e^{2 \mu b}-ae^{2 \mu a}}{\displaystyle e^{2 \mu b}-e^{2 \mu a}} -\frac{1}{2\mu}.  
\end{equation*}
The cost function is
$c(x,y)= \vert  x - y \vert .$
We proceed to analyze the function $\rho$. The notations are the same as Proposition \ref{T:Equilibrium}. 
The equation $\pi_2(a,R(a,b))=\pi_1(b,R(a,b))$ is equivalent to
$$
F(a,b)=(b+ a) +\frac{1}{\mu} -2 \left( \frac{\displaystyle be^{2\mu b}-ae^{2\mu a}}{e^{2\mu b}-e^{2\mu a}}  \right) +\mu (q_u+q_d)=0.
$$
On one hand, when $a<0$ the equation $F(a,b)=0$ has a solution $b>0$ because
\[  
a+ \mu^{-1}+\mu(q_u+q_d)<0 
\]
On the other, when $a \geq 0,$ the equation $F(a,b)=0$ also has a root because $b-2R(a,b) \to -\infty$ when $b \to \infty$. \\
We compute the partial derivative
$$\frac{\partial F}{\partial b} (a,b)=  -1   \frac{\displaystyle (e^{2\mu (b-a)}) \left(  2\mu (a-b)-1+e^{2 \mu (b-a)}  \right)}{\displaystyle (1-e^{2\mu (b-a)})^2} > 0 . $$ 
We deduce that the function $\rho$ is well defined in all $\mathbb{R}$  and the roots of $F(a,b)$ are unique for each $a$.
Furthermore, if $C$ is the positive constant that satisfies the equality
\[C(1+e^{2 \mu C})(1-e^{2 \mu C})^{-1}=-(q_u+q_d) \mu - \mu^{-1}, \] then $F(a,a+C)=0$. So $\rho(a)= a+C$.

From  Theorem \ref{T:optimal} we know the equilibrium points $(a,b)$ must satisfy the equality:
\begin{equation}\label{E:Integralabsolutebeta1}
 \int_a^{b} (\vert t- R(a,b) \vert -b + R(a,b) )2e^{2\mu t}dt + (q_u+q_d)e^{2 \mu a}=0  .
\end{equation}
More explicitly,
\begin{align}\label{E:Integralabsolutebeta1two}
 \int_a^{b} \bigg(\Big\vert t+\frac{1}{2\mu} -\frac{\displaystyle be^{2\mu b}-ae^{2 \mu a} }{\displaystyle e^{2 \mu b} -e^{2\mu a}} \Big\vert -b-\frac{1}{2 \mu} + &\frac{\displaystyle be^{2\mu b}-ae^{2 \mu a} }{\displaystyle e^{2 \mu b} -e^{2\mu a}}  \bigg)2e^{2\mu t}dt \nonumber \\ &+ (q_u+q_d)e^{2 \mu a}=0  
\end{align}
With the change of variable $u=t-b$ the equality \eqref{E:Integralabsolutebeta1two} is equivalent to: 
\begin{align}\label{E:Integralabsolutebeta1three}
 \int_{a-b}^{0} \bigg(\Big\vert u-\frac{\displaystyle (b-a)}{\displaystyle e^{2\mu(b-a)}-1} +\frac{1}{2\mu} \Big\vert  + \frac{\displaystyle (b-a) }{\displaystyle e^{2 \mu (b-a)} -1 }-\frac{1}{2\mu} \bigg)&2e^{2\mu (u+b-a)}du \ e^{2 \mu a}\nonumber \\ &+ (q_u+q_d)e^{2 \mu a}=0 .
\end{align}
Therefore if there is a point $(A,B)$ that satisfies \eqref{E:Integralabsolutebeta1two} then every point $(a,b)$ such that $b-a=B-A $ also satisfies \eqref{E:Integralabsolutebeta1two}. 
To solve the integral define $C \coloneqq b-a$ and $K \coloneqq C (\exp(2 \mu C)-1)^{-1}-(2\mu)^{-1}$ so the integral in \eqref{E:Integralabsolutebeta1three} becomes
\begin{multline*}
\int_{-C}^0 \left( \left\vert u-K\right\vert + K\right) 2e^{2 \mu (C+u)}du\\
= 
2e^{2 \mu  K}\int_{-C-K}^{-K} \vert r \vert e^{2 \mu (C+r)}dr + Ke^{2 \mu C}\frac{1-e^{-2 \mu C }}{\mu}\\
= \Big(\frac{\displaystyle C}{\displaystyle e^{2 \mu C}-1 } \Big) \frac{\displaystyle 2 e^{2 \mu C}}{\displaystyle \mu} + \frac{\displaystyle -2e^{2 \mu C}+2C \mu +1}{\displaystyle 2 \mu^2}.
\end{multline*}
Solving the integral in \eqref{E:Integralabsolutebeta1two} we conclude that a point $(a,b)$ is an equilibrium point iff $C \coloneqq b-a$ satisfies 
\begin{align*}
 &C(1+e^{2 \mu C})(1-e^{2 \mu C})^{-1}+(q_u+q_d)\mu + \mu^{-1}=0 \\ 
 &\Big(\frac{\displaystyle C}{\displaystyle e^{2 \mu C}-1 } \Big) \frac{\displaystyle 2 e^{2 \mu C}}{\displaystyle \mu} + \frac{\displaystyle -2e^{2 \mu C}+2C \mu +1}{\displaystyle 2 \mu^2} +q_d +q_u=0
\end{align*}
Using \ref{T:Limitcost} it can be shown that the value of the game is the same for all equilibrium points, and it is
\[ 
-C \exp(-2 \mu C)(1-\exp(-2 \mu C))^{-1}-(2\mu)^{-1}+q_d \mu. 
\]
\subsubsection{Multiplicative cost}\label{S:SubSubsectionOUMC}
Proposition \ref{T:Equilibrium} is used with $a_0=0$. We assume $\lambda \geq 1$ (in other cases symmetrical arguments can be used), 
$c(x,y)= \max(-\lambda x ,x ) (1+ \vert y \vert^{\beta})$, $q:= q_d=q_u \lambda^{-1}, \ \beta \geq 1$ and  $q_d \theta < 1 \wedge \lambda^{-1} $ so the function $c$ is a cost function.

The equality $\pi_1(b,R(a,b))=\pi_2(a,R(a,b))$ reads as $-\lambda a =b$ taking into account that $a<0<b$ . Furthermore if $\sigma$ is an even function we deduce $R(a,-\lambda a)$ is decreasing in $a$ so:
$$
\pi_2 (t,L(a))- \pi_2(a,L(a))= (a-t)(1+R(a,-\lambda a)^{\beta}-q_d \lambda \theta),
$$
which decreases to $-\infty $ in $a$. 
$$
\pi_1 (t,L(a))- \pi_1(-\lambda a,L(a))= (t- \lambda a)(1+R(a,-\lambda a)^{\beta}-q_u \theta),
$$
which decrease to $-\infty$ implying that uniqueness and existence holds.


\begin{thebibliography}{}
\bibitem[A{\"\i}d et al.(2023)]{ABF}
A{\"\i}d, R. and Basei, M. and Ferrari, G: 
A Stationary Mean-Field Equilibrium Model of Irreversible Investment in a Two-Regime Economy. 
\url{https://doi.org/10.48550/arXiv.1803.03464} (2023)
\bibitem[Alvarez(2018)]{Alvarez}
Alvarez, L.H.R.: 
A Class of Solvable Stationary Singular Stochastic Control Problems. 
\url{https://doi.org/10.48550/arXiv.1803.03464} (2018)
\bibitem[Alvarez and Shepp(1998)]{AS}
Alvarez, L. H. R., Shepp, L. A.:
Optimal harvesting of stochastically fluctuating populations. 
Journal of Mathematical Biology 37, 155--177 (1998)
\bibitem[Asmussen and Taksar(1997)]{AT}
Asmussen, S., Taksar, M.:
Controlled diffusion models for optimal dividend pay-out.
Insurance: Mathematics and Economics 20, 1--15 (1998)
\bibitem[Borodin and Salminen(2002)]{BS}
Borodin, A., Salminen, P.:
Handbook on Brownian motion - facts and formulae, 2nd ed. (2nd printing). 
Birkh\"auser, Basel (2002)
\bibitem[Cao et al.(2023)]{CDF}
Cao, H., Dianetti, J., Ferrari, G.: 
Stationary Discounted and Ergodic Mean Field Games of Singular Control. 
Mathematics of Operations Research 48, 1871--1898 (2023)
\bibitem[Cao and Guo(2022)]{CG}
Cao, H., Guo, X.:
MFGs for partially reversible investment.
Stochastic Processes and their Applications 150, 995--1014 (2022)
\bibitem[Carmona(2021)]{Carmona}
Carmona, R.: 
Applications of mean field games in financial engineering and economic theory.
Mean field games, 
Proceedings of Symposia in Applied Mathematics, 78, 165--219 (2021)
\bibitem[Carmona and Delarue(2013)]{carmona2013}
Carmona, R. and Delarue, F.: 
Probabilistic analysis of mean-field games.
SIAM Journal of control and optimization, 51, 2705--2734 (2013)
\bibitem[Carmona and Delarue(2018)]{CD(2018)}
Carmona, R. and Delarue, F.: 
Probabilistic theory of mean field games with applications.
Springer, 2018.
\bibitem[Christensen et al.(2021)]{SBT}
Christensen,  S., Neumann, B. A., Sohr, T.: 
Competition versus Cooperation: A class of solvable mean field impulse control problems.
SIAM Journal on Control and Optimization 59, 3946--3972 (2021)
\bibitem[Dianetti et. al. (2019)]{DFFN}
Dianetti, J., Ferrari, G., Fischer, M. and Nendel, M.:
Submodular Mean Field Games: Existence and Approximation of Solutions.
Ann. Appl. Probab. 31 (6) 2538 - 2566, (2021). 

\bibitem[Dianetti et. al. (2023)]{DFT}
Dianetti, J., Ferrari, G., and Tzouanas, I.:
Ergodic Mean-Field Games of Singular Control with Regime-Switching (extended version).
arXiv preprint, arXiv:2307.12012 (2023). 
\bibitem[Festa and Göttlich(2018)]{ADRIANOFIESTA!}
Festa, A., G\"ottlich, S.:
A Mean Field Game approach for multi-lane traffic management.
IFAC-PapersOnLine 51(32), 793--798 (2018)
\bibitem[Hening et al.(2019)]{HNUW}
Hening, A., Nguyen, D. H., Ungureanu, S. C., and Wong, T. K.: 
Asymptotic harvesting of populations in random environments. 
Journal of Mathematical Biology 78, 293--329 (2019)
\bibitem[H$\o$jgaard  and Taksar(2001)]{HT}
H$\o$jgaard, B., Taksar, M.:
Optimal risk control for a large corporation in the presence of returns on investments. 
Finance and Stochastics 5, 527--547 (2001)
\bibitem[Huang(2013)]{M.H(2013)}
Huang, M.: 
A Mean Field Capital Accumulation Game with HARA Utility. 
Dynamic Games and Applications 3, 446--472 (2013) 
\bibitem[Huang et al.(2004)]{MCR(2004)}
Huang, M., Caines, P. E., Malhame, R. P.: 
Large-population cost-coupled LQG problems: generalizations to non-uniform individuals.
2004 43rd IEEE Conference on Decision and Control, 3453--3458 (2004)
\bibitem[Huang et al.(2006)]{HMC(2006)}
Huang, M., Malhame, R. P., Caines, P. E.: 
Large population stochastic dynamic games: Closed-loop McKean-Vlasov systems and the Nash certainty equivalence principle. 
Communications in Information and Systems, 6(3), 221--252 (2006)
\bibitem[Jeanblanc-Picqué  and Shiryaev(1995)]{JS}
Jeanblanc-Picqué, M. and Shiryaev, A. N.: 
Optimization of the flow of dividends. 
Russian Mathematical Surveys 50, 257--277 (1995)
\bibitem[Kunwai et al.(2022)]{KKXYZ}
Kunwai, K., Xi, F., Yin, G., Zhu, C.: 
On an Ergodic Two-Sided Singular Control Problem. 
Applied Mathematics and Optimization 86, 26 (2022) 
\bibitem[Lande et al.(1994)]{LESE}
Lande, R., Engen S., Sæther B. E.: 
Optimal harvesting, economic discounting and extinction risk in fluctuating populations. 
Nature 372, 88--90 (1994)
\bibitem[Lasry and Lions(2007)]{LL(2007)}
Lasry, J-M., Lions, P-L.: 
Mean field games. Japanese Journal of Mathematics 2, 229--260 (2007) 
\bibitem[Paulsen(2008)]{P}
Paulsen, J.: 
Optimal dividend payments and reinvestments of diffusion processes with fixed and proportional costs. 
SIAM Journal on Control and Optimization 47, 2201--2226 (2008)
\bibitem[Peura  and Keppo(2006)]{BK}
Peura, S., Keppo, J. S.: 
Optimal Bank Capital with Costly Recapitalization.
Journal of Business 79, 2163--2201 (2006)
\bibitem[Protter(2005)]{ProtterPE}
Protter, P. E.:
Stochastic integration and differential equations, 2nd. edition.  
Springer, Berlin Heildeberg (2005)
\bibitem[Rogers and Williams(2000)]{RW}
Rogers, L.C.G., Williams, D.: 
Diffusions, Markov processes and Martingales, Volume 2: It$\hat{\text{o}}$ Calculus, 2nd edition. 
Cambridge University Press, Cambridge (2000)
\bibitem[Saisho(1987)]{saisho}
Saisho, Y.: 
Stochastic differential equations for multidimensional domain with reflecting boundary.
Probability Theory and Related Fields  74, 455--477  (1987)
\bibitem[Shreve et al.(1984)]{SLG} 
Shreve, S., Lehoczky, J., Gaver, D.: 
Optimal consumption for general diffusion with absorbing and reflecting barriers. 
SIAM Journal on Control and Optimization 22, 55--75 (1984)
\end{thebibliography}
\end{document}